\documentclass[11pt]{article}
\usepackage[]{amsmath,amssymb}
\usepackage{cite}
\newtheorem{theorem}{Theorem}[section]
\newtheorem{lemma}[theorem]{Lemma}
\newtheorem{proposition}[theorem]{Proposition}
\newtheorem{corollary}[theorem]{Corollary}
\newtheorem{exAux}[theorem]{Example}

\newtheorem{Def}[theorem]{Definition}
\newenvironment{definition}{\begin{Def} \rm}{\end{Def}}
\newtheorem{Note}[theorem]{Note}
\newenvironment{note}{\begin{Note} \rm}{\end{Note}}
\newtheorem{Problem}[theorem]{Problem}

\newtheorem{Rem}[theorem]{Remark}

\newtheorem{Not}[theorem]{Notation}
\newenvironment{notation}{\begin{Not} \rm}{\end{Not}}
\newtheorem{Conj}[theorem]{Conjecture}
\newenvironment{conjecture}{\begin{Conj}}{\end{Conj}}
\newtheorem{Ass}[theorem]{Assumption}

\newenvironment{proof}{\medskip\noindent{\bf Proof.\ }}{\qed\medskip}
\newenvironment{proofof}[1]{\medskip\noindent{\bf Proof  of {#1}.\ 
}}{\qed\medskip}
\newcommand{\qed}{\hfill\mbox{$\Box$\qquad\qquad}}

\renewcommand{\th}{\theta}

\renewcommand{\indent}{\hspace{6mm}}

\begin{document}
\thispagestyle{empty}

\begin{center}
\LARGE \bf
\noindent
Towards a classification of the tridiagonal pairs
\end{center}

\smallskip

\begin{center}
\Large
Kazumasa Nomura and 
Paul Terwilliger\footnote{This author gratefully acknowledges 
support from the FY2007 JSPS Invitation Fellowship Program
for Reseach in Japan (Long-Term), grant L-07512.}
\end{center}

\smallskip

\begin{quote}
\small 
\begin{center}
\bf Abstract
\end{center}

\indent
Let $\mathbb{K}$ denote a field and let $V$ denote a vector space
over $\mathbb{K}$ with finite positive dimension.
Let $\text{End}(V)$ denote the $\mathbb{K}$-algebra
consisting of all $\mathbb{K}$-linear transformations
from $V$ to $V$.
We consider a pair $A,A^* \in \text{End}(V)$ 
that satisfy (i)--(iv) below:
\begin{itemize}
\item[(i)]
Each of $A,A^*$ is diagonalizable.
\item[(ii)]
There exists an ordering $\{V_i\}_{i=0}^d$ of the eigenspaces of 
$A$ such that
$A^* V_i \subseteq V_{i-1} + V_{i} + V_{i+1}$ for $0 \leq i \leq d$,
where $V_{-1}=0$ and $V_{d+1}=0$.
\item[(iii)]
There exists an ordering $\{V^*_i\}_{i=0}^\delta$ 
of the eigenspaces of $A^*$ such that
$A V^*_i \subseteq V^*_{i-1} + V^*_{i} + V^*_{i+1}$ for $0 \leq i \leq \delta$,
where $V^*_{-1}=0$ and $V^*_{\delta+1}=0$.
\item[(iv)] 
There is no subspace $W$ of $V$ such that
$AW \subseteq W$, $A^* W \subseteq W$, $W \neq 0$, $W \neq V$.
\end{itemize}
We call such a pair a {\em tridiagonal pair} on $V$.
Let $E^*_0$ denote the element of $\text{End}(V)$
such that $(E^*_0-I)V^*_0=0$ and
$E^*_0V^*_i=0$ for $1 \leq i \leq d$.
Let $\cal D$ (resp. ${\cal D}^*$) denote the
$\mathbb{K}$-subalgebra of $\text{End}(V)$ generated by $A$ (resp. $A^*$).
In this paper we prove that
the span of
$E^*_0{\cal D}{\cal D}^*{\cal D}E^*_0$
equals the span of $E^*_0{\cal D}E^*_0{\cal D}E^*_0$,
and that the elements of $E^*_0{\cal D}E^*_0$ mutually commute.
We relate these results to some conjectures
of Tatsuro Ito and the second author that
are expected to play a role in the classification
of tridiagonal pairs.
\end{quote}

\section{Tridiagonal pairs}

\indent
Throughout the paper $\mathbb{K}$  denotes a field
and $V$ denotes a vector space over $\mathbb{K}$
with finite positive dimension.

\medskip

We begin by recalling the notion of a tridiagonal pair.
We will use the following terms.
Let $\text{End}(V)$ denote the $\mathbb{K}$-algebra
consisting of all $\mathbb{K}$-linear transformations
from $V$ to $V$.
For $A \in \text{End}(V)$ and for a subspace $W \subseteq V$,
we call $W$ an {\em eigenspace} of $A$ 
whenever $W \neq 0$ and there exists $\theta \in \mathbb{K}$
such that $W=\{v \in V \,|\, Av=\theta v\}$; 
in this case $\th$ is
the {\em eigenvalue} of $A$ associated with $W$.
We say $A$ is {\em diagonalizable} whenever $V$ is spanned by
the eigenspaces of $A$.

\medskip

\begin{definition} \cite{ITT}     \label{def:TDpair}   \samepage
By a {\em tridiagonal pair} on $V$ we mean an ordered pair
$A,A^* \in \text{End}(V)$ that satisfy (i)--(iv) below:
\begin{itemize}
\item[(i)] 
Each of $A,A^*$ is diagonalizable.
\item[(ii)] 
There exists an ordering $\{V_i\}_{i=0}^d$ of the eigenspaces of 
$A$ such that
\begin{equation}              \label{eq:Astrid}
  A^* V_i \subseteq V_{i-1} + V_{i} + V_{i+1}  \qquad\qquad (0 \leq i \leq d),
\end{equation}
 where $V_{-1}=0$ and $V_{d+1}=0$.
\item[(iii)] 
There exists an ordering $\{V^*_i\}_{i=0}^\delta$ 
of the eigenspaces of $A^*$ such that
\begin{equation}           \label{eq:Atrid}
A V^*_i \subseteq V^*_{i-1} + V^*_{i} + V^*_{i+1} \qquad\qquad
            (0 \leq i \leq \delta),
\end{equation}
where $V^*_{-1}=0$ and $V^*_{\delta+1}=0$.
\item[(iv)] 
There is no subspace $W$ of $V$ such that
$AW \subseteq W$, $A^* W \subseteq W$, $W \neq 0$, $W \neq V$.
\end{itemize}
\end{definition}

\begin{note}      \label{note:star}        \samepage
It is a common notational convention to use $A^*$ to represent the
conjugate-transpose of $A$. We are not using this convention.
In a tridiagonal pair $A,A^*$ the linear transformations $A$ and
$A^*$ are arbitrary subject to (i)--(iv) above.
\end{note}

\medskip

Let $A,A^*$ denote a tridiagonal pair on $V$, 
as in Definition \ref{def:TDpair}.
By \cite[Lemma 4.5]{ITT} the integers $d$ and $\delta$ from (ii), (iii) are equal;
we call this common value the {\em diameter} of the pair. 

\medskip

We refer the reader to
\cite{AC,AC2,AC3,Bas,ITT,IT:shape,IT:uqsl2hat,IT:Krawt,N:aw,N:refine,N:height1,
NT:tde,NT:sharp,Vidar} for background on tridiagonal pairs.
See 
\cite{BI,BT:Borel,BT:loop,Bow,Ca,CaMT,CaW,Egge,F:RL,H:tetra,HT:tetra,
IT:non-nilpotent,IT:tetra,
IT:inverting,IT:drg,IT:loop,ITW:equitable,Leo,Koe,Mik,Mik2,
P,R:multi,R:6j,Suz,T:sub1,T:sub3,
T:qSerre,T:Kac-Moody,Z}
for related topics.
The following special case has received a lot of attention. 
A tridiagonal pair $A,A^*$ is called a {\em Leonard pair} whenever
the eigenspaces for $A$ (resp. $A^*$) all have dimension $1$.
See 
\cite{Cur:mlt,Cur:spinLP,H,M:LT,NT:balanced,NT:formula,NT:det,NT:mu,
NT:span,NT:switch,NT:affine,NT:maps,T:Leonard,T:24points,T:canform,T:intro,
T:intro2,T:split,T:array,T:qRacah,T:survey,TV,V,V:AW}
for information about Leonard pairs.

\section{Tridiagonal systems}

\indent
When working with a tridiagonal pair, it is often convenient to consider
a closely related object called a tridiagonal system.
To define a tridiagonal system, we recall a few concepts from linear
algebra.
Let $A$ denote a diagonalizable element of $\text{End}(V)$.
Let $\{V_i\}_{i=0}^d$ denote an ordering of the eigenspaces of $A$
and let $\{\th_i\}_{i=0}^d$ denote the corresponding ordering of
the eigenvalues of $A$.
For $0 \leq i \leq d$ let $E_i:V \to V$ denote the linear transformation
such that $(E_i-I)V_i=0$ and $E_iV_j=0$ for $j \neq i$ $(0 \leq j \leq d)$.
Here $I$ denotes the identity of $\text{End}(V)$.
We call $E_i$ the {\em primitive idempotent} of $A$ corresponding to $V_i$
(or $\th_i$).
Observe that
(i) $\sum_{i=0}^d E_i = I$;
(ii) $E_iE_j=\delta_{i,j}E_i$ $(0 \leq i,j \leq d)$;
(iii) $V_i=E_iV$ $(0 \leq i \leq d)$;
(iv) $A=\sum_{i=0}^d \th_iE_i$.
Moreover
\begin{equation}         \label{eq:defEi}
  E_i=\prod_{\stackrel{0 \leq j \leq d}{j \neq i}}
          \frac{A-\th_jI}{\th_i-\th_j}.
\end{equation}
We note that each of $\{E_i\}_{i=0}^d$,
$\{A^i\}_{i=0}^d$ is a basis for the $\mathbb{K}$-subalgebra
of $\text{End}(V)$ generated by $A$.

Now let $A,A^*$ denote a tridiagonal pair on $V$.
An ordering of the eigenspaces of $A$ (resp. $A^*$)
is said to be {\em standard} whenever it satisfies 
\eqref{eq:Astrid} (resp. \eqref{eq:Atrid}). 
We comment on the uniqueness of the standard ordering.
Let $\{V_i\}_{i=0}^d$ denote a standard ordering of the eigenspaces of $A$.
Then the ordering $\{V_{d-i}\}_{i=0}^d$ is standard and no other ordering
is standard.
A similar result holds for the eigenspaces of $A^*$.
An ordering of the primitive idempotents of $A$ (resp. $A^*$)
is said to be {\em standard} whenever
the corresponding ordering of the eigenspaces of $A$ (resp. $A^*$)
is standard.

\medskip

\begin{definition}          \label{def:TDsystem}  \samepage
By a {\em tridiagonal system} on $V$ we mean a sequence
\[
  \Phi=(A;\{E_i\}_{i=0}^d;A^*;\{E^*_i\}_{i=0}^d)
\]
that satisfies (i)--(iii) below.
\begin{itemize}
\item[(i)]
$A,A^*$ is a tridiagonal pair on $V$.
\item[(ii)]
$\{E_i\}_{i=0}^d$ is a standard ordering
of the primitive idempotents of $A$.
\item[(iii)]
$\{E^*_i\}_{i=0}^d$ is a standard ordering
of the primitive idempotents of $A^*$.
\end{itemize}
\end{definition}

\medskip

We will use the following notation.

\medskip

\begin{notation}   \label{notation}   \samepage
let $\Phi=(A;\{E_i\}_{i=0}^d;A^*;\{E^*_i\}_{i=0}^d)$ denote
a tridiagonal system on $V$.
We denote by $\cal D$ (resp. ${\cal D}^*$)
the $\mathbb{K}$-subalgebra of $\text{End}(V)$ generated by
$A$ (resp. $A^*)$.
For $0 \leq i \leq d$ let $\th_i$ (resp. $\th^*_i$)
denote the eigenvalue of $A$ (resp. $A^*$)
associated with the eigenspace $E_iV$ (resp. $E^*_iV$).
We observe $\{\th_i\}_{i=0}^d$ (resp. $\{\th^*_i\}_{i=0}^d$) are
mutually distinct and contained in $\mathbb{K}$.
\end{notation}

\medskip

Referring to Notation \ref{notation},
it has been conjectured that $E^*_0V$ has dimension $1$, 
provided that $\mathbb{K}$ is algebraically closed \cite{IT:shape}.
A more recent and stronger conjecture is that
$E^*_0{\cal T}E^*_0$ is commutative, where $\cal T$ is 
the $\mathbb{K}$-subalgebra of $\text{End}(V)$ generated by 
${\cal D}$ and ${\cal D}^*$ \cite{NT:sharp}.
There is more detailed version of this conjecture which reads as follows:

\medskip

\begin{conjecture}  {\rm \cite[Conjecture 12.2]{NT:sharp}} \label{conj:EDDE}
  \samepage
With reference to Notation {\rm \ref{notation}}
the following {\rm (i)}, {\rm (ii)} hold.
\begin{itemize}
\item[\rm (i)]
$E^*_0{\cal T}E^*_0$ is generated by $E^*_0{\cal D}E^*_0$.
\item[\rm (ii)]
The elements of $E^*_0 {\cal D} E^*_0$ mutually commute.
\end{itemize}
\end{conjecture}

\medskip

The following special case of Conjecture \ref{conj:EDDE} 
has been proven. 
Referring to Notation \ref{notation}, there is a well known parameter $q$ 
associated with $A,A^*$ that is used to describe the eigenvalues 
\cite{ITT,T:qSerre}.
In \cite{IT:aug}, Conjecture \ref{conj:EDDE}
is proven assuming $q$ is not a root of unity
and $\mathbb{K}$ is algebraically closed. 
In this paper we use a different approach to 
prove part (ii) of Conjecture \ref{conj:EDDE}
without any extra assumptions. 
We also obtain a result which sheds some light on
why part (i) of the conjecture should be true without
any extra assumptions. 
We now state our main theorem.

\medskip

\begin{theorem}   \label{thm:main}  \samepage
With reference to Notation {\rm \ref{notation}}
the following {\rm (i)}, {\rm (ii)} hold.
\begin{itemize}
\item[\rm (i)]
The span of $E^*_0{\cal D}{\cal D}^*{\cal D}E^*_0$
is equal to the span of $E^*_0{\cal D}E^*_0{\cal D}E^*_0$.
\item[\rm (ii)]
The elements of $E^*_0{\cal D}E^*_0$ mutually commute.
\end{itemize}
\end{theorem}

\medskip

Our proof of Theorem \ref{thm:main} appears in Section 11.
In Sections 3--10 we obtain some results that will
be used in the proof. Our point of departure
is the following observation.

\medskip

\begin{lemma}    \label{lem:supertrid}
With reference to Notation {\rm \ref{notation}}
the following {\rm (i)}, {\rm (ii)} hold
for $0 \leq i,j,k \leq d$.
\begin{itemize}
\item[\rm (i)]
$E^*_iA^kE^*_j=0$ if $k<|i-j|$. 
\item[\rm (ii)]
$E_i{A^*}^kE_j=0$ if $k<|i-j|$.
\end{itemize}
\end{lemma}

\begin{proof}
Routinely obtained using lines \eqref{eq:Astrid}, \eqref{eq:Atrid}
and Definition \ref{def:TDsystem}.
\end{proof}

\section{The $D_4$ action}

\indent
Let $\Phi=(A; \{E_i\}_{i=0}^d; A^*; \{E^*_i\}_{i=0}^d)$
denote a tridigonal system on $V$.
Then each of the following is a tridiagonal system on $V$:
\begin{align*}
\Phi^{*}  &:= 
       (A^*; \{E^*_i\}_{i=0}^d; A; \{E_i\}_{i=0}^d), \\
\Phi^{\downarrow} &:=
       (A; \{E_i\}_{i=0}^d; A^*; \{E^*_{d-i}\}_{i=0}^d), \\
\Phi^{\Downarrow} &:=
       (A; \{E_{d-i}\}_{i=0}^d; A^*; \{E^*_{i}\}_{i=0}^d).
\end{align*}
Viewing $*$, $\downarrow$, $\Downarrow$ as permutations on the set of
all tridiagonal systems,
\begin{equation}    \label{eq:relation1}
\text{$*^2$$\;=\;$$\downarrow^2$$\;=\;$$\Downarrow^2$$\;=\;$$1$},
\end{equation}
\begin{equation}    \label{eq:relation2}
\text{$\Downarrow$$*$$\;=\;$$*$$\downarrow$}, \quad
\text{$\downarrow$$*$$\;=\;$$*$$\Downarrow$}, \quad
\text{$\downarrow$$\Downarrow$$\;=\;$$\Downarrow$$\downarrow$}.
\end{equation}
The group generated by symbols $*$, $\downarrow$, $\Downarrow$ subject
to the relations (\ref{eq:relation1}), (\ref{eq:relation2}) is the
dihedral group $D_4$. We recall that $D_4$ is the group of symmetries of a
square, and has $8$ elements.
Apparently $*$, $\downarrow$, $\Downarrow$ induce an action of $D_4$
on the set of all tridiagonal systems.
Two tridiagonal systems will be called {\em relatives} whenever they are
in the same orbit of this $D_4$ action. 
The relatives of $\Phi$ are as follows:

\medskip
\noindent
\begin{center}
\begin{tabular}{c|c}
name  &  relative \\
\hline
$\Phi$ & 
       $(A; \{E_i\}_{i=0}^d; A^*;  \{E^*_i\}_{i=0}^d)$ \\ 
$\Phi^{\downarrow}$ &
       $(A; \{E_i\}_{i=0}^d; A^*;  \{E^*_{d-i}\}_{i=0}^d)$ \\ 
$\Phi^{\Downarrow}$ &
       $(A; \{E_{d-i}\}_{i=0}^d; A^*;  \{E^*_i\}_{i=0}^d)$ \\ 
$\Phi^{\downarrow \Downarrow}$ &
       $(A; \{E_{d-i}\}_{i=0}^d; A^*;  \{E^*_{d-i}\}_{i=0}^d)$ \\ 
$\Phi^{*}$  & 
       $(A^*; \{E^*_i\}_{i=0}^d; A;  \{E_i\}_{i=0}^d)$ \\ 
$\Phi^{\downarrow *}$ &
       $(A^*; \{E^*_{d-i}\}_{i=0}^d; A;  \{E_i\}_{i=0}^d)$ \\ 
$\Phi^{\Downarrow *}$ &
       $(A^*; \{E^*_i\}_{i=0}^d; A;  \{E_{d-i}\}_{i=0}^d)$ \\ 
$\Phi^{\downarrow \Downarrow *}$ &
       $(A^*; \{E^*_{d-i}\}_{i=0}^d; A;  \{E_{d-i}\}_{i=0}^d)$
\end{tabular}
\end{center}

\section{Some polynomials}

\indent
Let $\lambda$ denote an indeterminate and 
let $\mathbb{K}[\lambda]$ denote the $\mathbb{K}$-algebra
consisting of all polynomials in $\lambda$ that have
coefficients in $\mathbb{K}$.

\medskip

\begin{definition}   \label{def:tau}
With reference to Notation \ref{notation},
for $0 \leq i \leq d$ we define the following polynomials
in $\mathbb{K}[\lambda]$:
\begin{align}
 \tau_i
 &= (\lambda-\th_0)(\lambda-\th_1)\cdots(\lambda-\th_{i-1}),
                                                    \label{eq:deftau}  \\
 \eta_i
 &= (\lambda-\th_d)(\lambda-\th_{d-1})\cdots(\lambda-\th_{d-i+1}),
                                                    \label{eq:defeta}  \\
 \tau^*_i
  &= (\lambda-\th^*_0)(\lambda-\th^*_1)\cdots(\lambda-\th^*_{i-1}),
                                                    \label{eq:deftaus}  \\
 \eta^*_i
  &= (\lambda-\th^*_d)(\lambda-\th^*_{d-1})\cdots(\lambda-\th^*_{d-i+1}).
                                                    \label{eq:defetas}
\end{align}
Note that each of $\tau_i,\eta_i,\tau^*_i,\eta^*_i$ is monic with degree $i$.
\end{definition}

\medskip

The following lemmas show the significance of these polynomials.
We will focus on $\tau_i,\eta_i$;
of course similar results hold for $\tau^*_i,\eta^*_i$.

\medskip

\begin{lemma}   \label{lem:basistauiA}   \samepage
With reference to Notation {\rm \ref{notation}},
each of $\{\tau_i(A)\}_{i=0}^d$, $\{\eta_i(A)\}_{i=0}^d$
form a basis for $\cal D$.
\end{lemma}

\begin{proof}
The sequence $\{A^i\}_{i=0}^d$ is a basis for $\cal D$ and
each of $\tau_i, \eta_i$ has degree $i$ for $0 \leq i \leq d$.
\end{proof}

\begin{lemma}   \label{lem:tauiAEi}  \samepage
With reference to Notation {\rm \ref{notation}},
for $0 \leq i \leq d$ we have
\begin{align}              
  \tau_i(A) &= \sum_{j=i}^d \tau_i(\th_j)E_j,
& E_i &= \sum_{j=i}^d
          \frac{\eta_{d-j}(\th_i)\tau_j(A)}
               {\tau_i(\th_i)\eta_{d-i}(\th_i)},  \label{eq:tauiA}  \\
  \eta_i(A) &= \sum_{j=0}^{d-i} \eta_i(\th_{j}) E_{j},
& E_i &= \sum_{j=d-i}^{d}
             \frac{\tau_{d-j}(\th_{i})\eta_{j}(A)}
                  {\tau_i(\th_i)\eta_{d-i}(\th_i)}.  \label{eq:etaiA}
\end{align}
\end{lemma}

\begin{proof}
To get the equation on the left in \eqref{eq:tauiA}, observe that
\[
  \tau_i(A) = \sum_{j=0}^d \tau_i(A)E_j 
            = \sum_{j=0}^d \tau_i(\th_j)E_j
\]
and $\tau_i(\th_j)= 0$ for $0 \leq j \leq i-1$.
Concerning the equation on the right in \eqref{eq:tauiA}, 
first observe by \eqref{eq:defEi} that
\begin{equation}    \label{eq:Eiaux}
 E_i = \frac{\tau_i(A)\eta_{d-i}(A)}
            {\tau_i(\th_i)\eta_{d-i}(\th_i)}.
\end{equation}
By \cite[Lemma 5.4]{NT:mu} or a routine induction on $d$ we find
\[
  \eta_{d-i} = \sum_{j=i}^d \eta_{d-j}(\th_i)\tau^{-1}_i \tau_j.
\]
Therefore
\begin{equation}   \label{eq:etaaux}
 \tau_i \eta_{d-i} = \sum_{j=i}^d \eta_{d-j}(\th_i) \tau_j.
\end{equation}
Evaluating the right-hand side of \eqref{eq:Eiaux} using
\eqref{eq:etaaux} we obtain the equation on the right in \eqref{eq:tauiA}.
We have now obtained \eqref{eq:tauiA}.
Applying \eqref{eq:tauiA} to $\Phi^{\Downarrow}$ we obtain \eqref{eq:etaiA}.
\end{proof}

\begin{lemma}   \label{lem:taubasis}   \samepage
With reference to Notation {\rm \ref{notation}} the following
{\rm (i)}--{\rm (iii)} hold for $0 \leq i \leq d$.
\begin{itemize}
\item[\rm (i)]
  $\text{\rm Span}\{A^h \,|\, 0 \leq h \leq i\} 
    = \text{\rm Span}\{\tau_h(A) \,|\, 0 \leq h \leq i\}$.
\item[\rm (ii)]
  $\text{\rm Span}\{E_h \,|\, i \leq h \leq d\} 
  = \text{\rm Span}\{\tau_h(A) \,|\, i \leq h \leq d\}$.
\item[\rm (iii)] 
 $\tau_i(A)$ is a basis for 
  $ \text{\rm Span}\{A^h \,|\, 0 \leq h \leq i\} 
       \cap \text{\rm Span}\{E_h \,|\, i \leq h \leq d\}$.
\end{itemize}
\end{lemma}

\begin{proof}
(i):
Recall that $\tau_h$ has degree $h$ for $0 \leq h \leq d$.

(ii):
Follows from Lemma \ref{lem:tauiAEi}.

(iii):
Immediate from (i), (ii) above.
\end{proof}

\medskip

Applying Lemma \ref{lem:taubasis} to $\Phi^{\Downarrow}$ we obtain the
following result.

\medskip

\begin{lemma}   \label{lem:etabasis}   \samepage
With reference to Notation {\rm \ref{notation}} the following
{\rm (i)}--{\rm (iii)} hold for $0 \leq i \leq d$.
\begin{itemize}
\item[\rm (i)]
  $\text{\rm Span}\{A^h \,|\, 0 \leq h \leq i\} 
    = \text{\rm Span}\{\eta_h(A) \,|\, 0 \leq h \leq i\}$.
\item[\rm (ii)]
  $\text{\rm Span}\{E_h \,|\, 0 \leq h \leq d-i \} 
  = \text{\rm Span}\{\eta_h(A) \,|\, i \leq h \leq d\}$.
\item[\rm (iii)] 
 $\eta_i(A)$ is a basis for 
   $\text{\rm Span}\{A^h \,|\, 0 \leq h \leq i\} 
       \cap \text{\rm Span}\{E_h \,|\, 0 \leq h \leq d-i \}$.
\end{itemize}
\end{lemma}

\section{Some bases for $\cal D$ and ${\cal D}^*$}

\indent
In this section we give some bases for $\cal D$ and ${\cal D}^*$
that will be useful later in the paper.
We will state our results for $\cal D$; 
of course similar results hold for ${\cal D}^*$.

\medskip

\begin{lemma}   \label{lem:replace}
With reference to Notation {\rm \ref{notation}} consider 
the following basis for $\cal D$:
\begin{equation}   \label{eq:E0Ed}
  E_0,E_1,\ldots,E_d.
\end{equation}
For $0 \leq n \leq d$, if we replace any $(n+1)$-subset of \eqref{eq:E0Ed}
by $I,A,A^2, \ldots, A^n$ then the result is still a basis for $\cal D$.
\end{lemma}

\begin{proof}
Let $\Delta$ denote a $(n+1)$-subset of $\{0,1,\ldots, d\}$
and let $\overline{\Delta}$ denote the complement of
$\Delta$ in $\{0,1,\ldots, d\}$. We show
\begin{equation}                  \label{eq:basisaux}     
   \{A^i\}_{i=0}^n \cup \{E_i\}_{i \in \overline{\Delta}}       
\end{equation}
is a basis for $\cal D$. The number of elements in \eqref{eq:basisaux}
is $d+1$ and this equals the dimension of $\cal D$. Therefore it
suffices to show the elements \eqref{eq:basisaux} span $\cal D$. 
Let $S$ denote the subspace of $\cal D$ spanned by \eqref{eq:basisaux}. 
To show ${\cal D} = S$ we show $E_i \in S$ for $i \in \Delta$. 
For $0 \leq j \leq n$ we have $A^j = \sum_{i=0}^d \theta^j_i E_i$.
In these equations we rearrange the terms to find
\begin{equation}           \label{eq:basisaux2}
  \sum_{i \in \Delta} \th^j_i E_i \in  S \qquad\qquad  (0 \leq j \leq n).
\end{equation}
In the linear system \eqref{eq:basisaux2} the coefficient matrix is 
Vandermonde and hence nonsingular. Therefore $E_i \in S$ for $i \in \Delta$.
Now $S={\cal D}$ and the result follows.
\end{proof}

\section{The space $R$}

\begin{definition}    \label{def:pi}   \samepage
With reference to Notation \ref{notation} we consider
the tensor product 
${\cal D}\otimes{\cal D}^*\otimes{\cal D}$
where $\otimes = \otimes_{\mathbb{K}}$.
We define a $\mathbb{K}$-linear transformation
\[
 \pi : \qquad
 \begin{array}{ccc}
  {\cal D}\otimes{\cal D}^*\otimes{\cal D}  & \qquad \to \qquad &  \text{End}(V) \\
   X \otimes Y \otimes Z  & \qquad \mapsto \qquad & E^*_0XYZE^*_0
 \end{array}
\]
We note that the image of $\pi$ is the span of 
$E^*_0{\cal D}{\cal D}^*{\cal D}E^*_0$.
\end{definition}

\begin{definition}    \label{def:R}  \samepage
With reference to Notation \ref{notation} let $R$
denote the sum of the following three subspaces of
${\cal D} \otimes {\cal D}^* \otimes {\cal D}$:
\begin{equation}        \label{eq:L}
 \text{Span}\{A^i \otimes E^*_j \,|\, 0 \leq i<j \leq d\}
    \otimes {\cal D},
\end{equation}
\begin{equation}        \label{eq:R} 
{\cal D} \otimes 
 \text{Span}\{E^*_j \otimes A^i \,|\, 0 \leq i<j \leq d\}, 
\end{equation}
\begin{equation}        \label{eq:M}
 \text{Span}\{E_i \otimes A^{*t} \otimes E_j
                    \,|\, 0 \leq i,j,t \leq d, \, t<|i-j| \}.
\end{equation}
\end{definition}

\begin{lemma}   \label{lem:kernel}   \samepage
With reference to Definitions {\rm \ref{def:pi}} and {\rm \ref{def:R}}
the space $R$ is contained in the kernel of $\pi$.
\end{lemma}

\begin{proof}
Routinely obtained using Lemma \ref{lem:supertrid}.
\end{proof}

\medskip

With reference to Notation \ref{notation} and Lemma \ref{lem:kernel},
we desire to understand the kernel of $\pi$.
To gain this understanding we systematically investigate $R$.
We proceed as follows.

\medskip

\begin{lemma}    \label{lem:DDsD}     \samepage
With reference to Notation {\rm \ref{notation}},
\[
   {\cal D} \otimes {\cal D}^* \otimes {\cal D} 
      = \sum_{t=0}^d {\cal D} \otimes \tau^*_t(A^*) \otimes {\cal D}
  \qquad\qquad  (\text{\rm direct sum}).
\]
\end{lemma}

\begin{proof}
Applying Lemma \ref{lem:basistauiA}(i) to $\Phi^*$ we find that
$\{\tau^*_t(A^*)\}_{t=0}^d$ is a basis for ${\cal D}^*$.
The result follows.
\end{proof}

\begin{definition}    \label{def:Rt}    \samepage
With reference to Notation {\rm \ref{notation}} and Definition {\rm \ref{def:R}},
for $0 \leq t \leq d$ let $R_t$ denote the intersection of $R$ with
${\cal D}\otimes \tau^*_t(A^*) \otimes {\cal D}$.
\end{definition}

\medskip

With reference to Notation \ref{notation} and Definition \ref{def:R}, 
our next goal is to  show $R=\sum_{t=0}^d R_t$ (direct sum).
The following lemma will be useful.

\medskip

\begin{lemma}   \label{lem:coincidepre}   \samepage
With reference to Notation {\rm \ref{notation}}
the following {\rm (i)}--{\rm (iii)} hold.
\begin{itemize}
\item[\rm (i)]
The space \eqref{eq:L} is the direct sum over $t=0,1,\ldots,d$
of the following subspaces:
\begin{equation}     \label{eq:Lt}
 \text{\rm Span}\{A^i \,|\, 0 \leq i<t\}
   \otimes \tau^*_t(A^*) \otimes {\cal D}.
\end{equation}
\item[\rm (ii)]
The space  \eqref{eq:R} is the direct sum over $t=0,1,\ldots,d$
of the following subspaces:
\begin{equation}     \label{eq:Rt}
 {\cal D} \otimes \tau^*_t(A^*) \otimes \text{\rm Span}\{A^i \,|\, 0 \leq i<t\}.
\end{equation}
\item[\rm (iii)]
The space  \eqref{eq:M} is the direct sum over $t=0,1,\ldots,d$
of the following subspaces:
\begin{equation}    \label{eq:Mt}
 \text{\rm Span}\{E_i \otimes \tau^*_t(A^*) \otimes E_j \,|\,
                          0 \leq i,j \leq d,\, t<|i-j| \}.
\end{equation}
\end{itemize}
\end{lemma}

\begin{proof}
(i):
Applying Lemma \ref{lem:taubasis}(ii) to $\Phi^*$ we obtain
\begin{align*}
 \text{Span}\{A^i & \otimes E^*_t \, |\, 0 \leq i< t \leq d \} \\
 &= \sum_{i=0}^d A^i \otimes \text{Span}\{E^*_t \,|\, i<t \leq d \} \\
 &= \sum_{i=0}^d A^i \otimes \text{Span}\{\tau^*_t(A^*) \,|\, i<t\leq d \} \\
 &= \sum_{t=0}^d \text{Span}\{A^i \,|\, 0 \leq i<t \} \otimes \tau^*_t(A^*).
\end{align*}
In the above lines we tensor each term on the right by $\cal D$ to find
that the space \eqref{eq:L} is the sum over $t=0,1,...,d$
of the spaces \eqref{eq:Lt}. 
The sum is direct by Lemma \ref{lem:DDsD}.

(ii):
Similar to the proof of (i) above.

(iii):
Applying Lemma \ref{lem:taubasis}(i) to $\Phi^*$ we obtain
\begin{align*}
 \text{Span}\{E_i \otimes &A^{*t} \otimes E_j \,|\,
                     0 \leq i,j,t \leq d,\, t<|i-j| \}  \\
 &= \sum_{i=0}^d\sum_{j=0}^d
     E_i \otimes \text{Span}\{A^{*t}\,|\, 0 \leq t<|i-j|\} \otimes E_j \\
 &= \sum_{i=0}^d\sum_{j=0}^d
     E_i \otimes \text{Span}\{\tau^*_t(A^*)\,|\, 0\leq t<|i-j|\}\otimes E_j\\
 &= \sum_{t=0}^d \text{Span}\{E_i \otimes \tau^*_t(A^*) \otimes E_j
                                 \,|\, 0 \leq i,j \leq d,\, t<|i-j| \}.
\end{align*}
In other words the space \eqref{eq:M} is the sum over $t=0,1,\ldots,d$
of the spaces \eqref{eq:Mt}.
This sum is direct by Lemma \ref{lem:DDsD}.
\end{proof}

\begin{theorem}    \label{thm:coincide}   \samepage
With reference to Notation {\rm \ref{notation}} and 
Definition {\rm \ref{def:Rt}}
the following {\rm (i)}, {\rm (ii)} hold.
\begin{itemize}
\item[\rm (i)]
For $0 \leq t \leq d$ the subspace ${R}_t$ is the sum of
the spaces \eqref{eq:Lt}--\eqref{eq:Mt}.
\item[\rm (ii)]
${R} = \sum_{t=0}^d {R}_t$ (direct sum).
\end{itemize}
\end{theorem}

\begin{proof}
For $0 \leq t \leq d$ let ${R}'_t$ denote the sum of the spaces
\eqref{eq:Lt}, \eqref{eq:Rt}, \eqref{eq:Mt}. 
Note that ${R}'_t$ is contained in 
${\cal D} \otimes \tau^*_t(A^*) \otimes {\cal D}$,
and that 
${R} = \sum_{t=0}^d {R}'_t$ by Lemma \ref{lem:coincidepre}.
By these comments and Lemma  \ref{lem:DDsD} we find 
${R}'_t = {R}_t$  for $0 \leq t \leq d$.
The result follows.
\end{proof}

\section{A basis for $R_t$ and ${\cal D} \otimes \tau^*_t(A^*)\otimes{\cal D}$}

\indent
With reference to Notation \ref{notation} and Definition \ref{def:Rt},
for $0 \leq t \leq d$ we display a basis for $R_t$
and extend this to a basis for 
 ${\cal D} \otimes \tau^*_t(A^*)\otimes{\cal D}$.

\medskip

\begin{theorem}   \label{thm:basis1}  \samepage
With reference to Notation {\rm \ref{notation}}
and Definition {\rm \ref{def:Rt}},
for $0 \leq t \leq d$ the following sets of vectors together form a basis
for ${\cal D} \otimes \tau^*_t(A^*) \otimes {\cal D}$:
\begin{align}
 & \{A^i \otimes \tau^*_t(A^*) \otimes A^j \,|\,
              0 \leq i\leq d,\, 0 \leq j < t \},     \label{eq:T1}  \\
 & \{A^i \otimes \tau^*_t(A^*) \otimes A^j \,|\,
             0 \leq i < t,\, t \leq j \leq d \},       \label{eq:T2} \\
 & \{E_i \otimes \tau^*_t(A^*) \otimes E_j \,|\,
             0 \leq i,j \leq d,\, t<|i-j| \},        \label{eq:S3}  \\
 & \{E_i \otimes \tau^*_t(A^*) \otimes E_i  \,|\,
             0 \leq i \leq d-t\}.                  \label{eq:basis1}
\end{align} 
Moreover the sets \eqref{eq:T1}--\eqref{eq:S3} together form a basis for 
${R}_t$.
\end{theorem}

\begin{proof}
The span of \eqref{eq:T1}--\eqref{eq:S3} equals the span of
\eqref{eq:Lt}--\eqref{eq:Mt}
and this equals ${R}_t$ by Theorem \ref{thm:coincide}(i).
The dimension of 
${\cal D} \otimes \tau^*_t(A^*) \otimes {\cal D}$ is $(d+1)^2$.
The cardinality of the sets \eqref{eq:T1}--\eqref{eq:basis1} is
$t(d+1)$, $t(d-t+1)$, $(d-t)(d-t+1)$, $d-t+1$ respectively,
and the sum of these numbers is $(d+1)^2$.
Therefore the number of vectors in \eqref{eq:T1}--\eqref{eq:basis1}
is equal to the dimension of 
${\cal D} \otimes \tau^*_t(A^*) \otimes {\cal D}$.
Consequently to finish the proof it suffices to show
that \eqref{eq:T1}--\eqref{eq:basis1} together span 
${\cal D} \otimes \tau^*_t(A^*) \otimes {\cal D}$.
Let ${S}$ denote the span of \eqref{eq:T1}--\eqref{eq:basis1}.
We first claim that for $0 \leq i \leq d-t$, both
\[
   E_i \otimes \tau^*_t(A^*) \otimes {\cal D} \subseteq S, \qquad\qquad
   {\cal D} \otimes \tau^*_t(A^*) \otimes E_i \subseteq S.
\]
To prove the claim, by induction on $i$ we may assume
\begin{equation}   \label{eq:hypo}
   E_j \otimes \tau^*_t(A^*) \otimes {\cal D} \subseteq S, \qquad
 {\cal D} \otimes \tau^*_t(A^*) \otimes E_j \subseteq S 
 \qquad\qquad (0 \leq j < i).
\end{equation}
By Lemma \ref{lem:replace} the following vectors together form a basis 
for $\cal D$:
\[
  E_0,E_1,\ldots,E_{i-1}; \qquad
  E_i; \qquad
  I,A,A^2,\ldots,A^{t-1}; \qquad
  E_{i+t+1},E_{i+t+2},\ldots,E_d.
\]
Therefore $E_i \otimes \tau^*_t(A^*) \otimes {\cal D}$ is the sum of
the following spaces:
\begin{align}
 &E_i \otimes \tau^*_t(A^*) \otimes \text{Span}\{E_0,E_1,\ldots,E_{i-1}\},
                                                       \label{eq:space1} \\
 &E_i \otimes \tau^*_t(A^*) \otimes \text{Span}\{E_i\}, \label{eq:space2} \\
 &E_i \otimes \tau^*_t(A^*) \otimes \text{Span}\{I,A,A^2,\ldots,A^{t-1}\},
                                                        \label{eq:space3}  \\
 &E_i \otimes \tau^*_t(A^*) \otimes \text{Span}\{E_{i+t+1},E_{i+t+2},\ldots,E_d\}.
                                                         \label{eq:space4}
\end{align}
The space \eqref{eq:space1} is contained in $S$ by \eqref{eq:hypo},
the space \eqref{eq:space2} is contained in $S$ by \eqref{eq:basis1},
the space \eqref{eq:space3} is contained in $S$ by \eqref{eq:T1},
and the space \eqref{eq:space4} is contained in $S$ by \eqref{eq:S3}.
Therefore $E_i \otimes \tau^*_t(A^*) \otimes {\cal D}$
is contained in $S$.
Similarly one shows that
${\cal D} \otimes \tau^*_t(A^*) \otimes E_i$ is contained in $S$ and the
claim is proved. 
Next we claim that
$E_i \otimes \tau^*_t(A^*) \otimes {\cal D}$ is contained in $S$ for 
$d-t<i \leq d$.
By Lemma \ref{lem:replace} the following vectors together form a basis for 
$\cal D$:
\[
   E_0,E_1,\ldots,E_{d-t};  \qquad
   I,A,A^2,\ldots,A^{t-1}.
\]
Therefore $E_i \otimes \tau^*_t(A^*) \otimes {\cal D}$ is the sum of
the following spaces:
\begin{align}
 & E_i \otimes \tau^*_t(A^*) \otimes \text{Span}\{E_0,E_1,\ldots,E_{d-t}\}, 
                                                    \label{eq:space5} \\
 & E_i \otimes \tau^*_t(A^*) \otimes \text{Span}\{I,A,A^2,\ldots,A^{t-1}\}.
                                                    \label{eq:space6}
\end{align}
The space \eqref{eq:space5} is contained in $S$ by the first claim, 
and the space \eqref{eq:space6} is contained in $S$ by \eqref{eq:T1}. 
Therefore $E_i \otimes \tau^*_t(A^*) \otimes {\cal D}$ is contained in $S$
and the second claim is proved.
By the two claims and since $\{E_i\}_{i=0}^d$ is a basis for $\cal D$, 
we find ${\cal D} \otimes \tau^*_t(A^*) \otimes {\cal D}$ is contained in $S$.
In other words \eqref{eq:T1}--\eqref{eq:basis1} together span 
${\cal D} \otimes \tau^*_t(A^*) \otimes {\cal D}$ as desired.
The result follows.
\end{proof}

\begin{corollary}  \label{cor:dim}   \samepage
With reference to Notation {\rm \ref{notation}} and
Definition {\rm \ref{def:Rt}}
the following {\rm (i)}--{\rm (iv)} hold.
\begin{itemize}
\item[\rm (i)]
For $0 \leq t \leq d$ the dimension of $R_t$ is $d^2+d+t$.
\item[\rm (ii)]
For $0 \leq t \leq d$ the codimension of $R_t$ 
in ${\cal D} \otimes \tau^*_t(A^*) \otimes {\cal D}$ is $d-t+1$.
\item[\rm (iii)]
The dimension of $R$ is $d(d+1)(2d+3)/2$.
\item[\rm (iv)] 
The codimension of $R$ in 
${\cal D} \otimes {\cal D}^* \otimes {\cal D}$ is $(d+1)(d+2)/2$.
\end{itemize}
\end{corollary}

\begin{proof}
(i), (ii):
The dimension of ${\cal D} \otimes \tau^*_t(A^*) \otimes {\cal D}$
is $(d+1)^2$. 
By \eqref{eq:basis1} the codimension of $R_t$ in
${\cal D} \otimes \tau^*_t(A^*) \otimes {\cal D}$ is $d-t+1$. 
The result follows.

(iii):
Sum the dimension in (i) over $t=0,1,\ldots, d$.

(iv):
Sum the codimension in (ii) over $t=0,1,\ldots,d$.
\end{proof}

\section{The map $\ddagger$}

\begin{definition}        \label{def:dd}   \samepage
With reference to Notation {\rm \ref{notation}} 
we define a $\mathbb{K}$-linear transformation
\[
 \ddagger: \qquad
 \begin{array}{ccc}
  {\cal D}\otimes{\cal D}^* \otimes{\cal D} 
     & \qquad \to \qquad & {\cal D} \otimes {\cal D}^* \otimes{\cal D} \\
   X \otimes Y \otimes Z  & \qquad \mapsto \qquad & Z \otimes Y \otimes X
 \end{array}
\]
We call $\ddagger$ the {\em transpose map}.
We observe that $\ddagger$ is an involution.
\end{definition}

\begin{proposition}         \label{prop:dd}   \samepage
With reference to Notation {\rm \ref{notation}} and
Definition {\rm \ref{def:dd}}
the following {\rm (i)}--{\rm (iii)} hold.
\begin{itemize}
\item[\rm (i)]
$R$ is invariant under $\ddagger$.
\item[\rm (ii)]
For $0 \leq t \leq d$ the space 
${\cal D} \otimes \tau^*_t(A^*) \otimes {\cal D}$
is invariant under $\ddagger$.
\item[\rm (iii)]
For $0 \leq t \leq d$ the space $R_t$ is invariant under $\ddagger$.
\end{itemize}
\end{proposition}

\begin{proof}
(i):
By Definition \ref{def:R} the space $R$ is the sum of 
\eqref{eq:L}--\eqref{eq:M}.
The map $\ddagger$ exchanges \eqref{eq:L}, \eqref{eq:R}
and leaves \eqref{eq:M} invariant. 
The result follows.

(ii):
Clear.

(iii):
By Theorem \ref{thm:coincide}(i) the space $R_t$ is the sum of
\eqref{eq:Lt}--\eqref{eq:Mt}.
The map $\ddagger$ exchanges \eqref{eq:Lt}, \eqref{eq:Rt} and
leaves \eqref{eq:Mt} invariant. 
The result follows.
\end{proof}

\begin{theorem}  \label{thm:dd}  \samepage
With reference to 
 Notation {\rm \ref{notation}} and
Definition {\rm \ref{def:dd}}
the following {\rm (i)}, {\rm (ii)} hold.
\begin{itemize}
\item[\rm (i)]
For $0 \leq t \leq d$ the image of 
${\cal D} \otimes \tau^*_t(A^*) \otimes {\cal D}$ 
under $1-\ddagger$ is contained in $R_t$.
\item[\rm (ii)]
The image of ${\cal D} \otimes {\cal D}^* \otimes {\cal D}$ under 
$1-\ddagger$ is contained in $R$.
\end{itemize}
\end{theorem}

\begin{proof}
(i): 
Let $C$ denote the subspace of 
${\cal D} \otimes \tau^*_t(A^*) \otimes {\cal D}$
spanned by the elements \eqref{eq:basis1}. 
By Theorem \ref{thm:basis1} the space
${\cal D} \otimes \tau^*_t(A^*) \otimes {\cal D}$
is the direct sum of $R_t$ and $C$. 
By Proposition \ref{prop:dd}(iii)
the image of $R_t$ under $1-\ddagger$ is contained in $R_t$. 
By \eqref{eq:basis1} the image of $C$ under $1-\ddagger$ is zero. 
The result follows.

(ii):
Combine Lemma \ref{lem:DDsD}, Theorem \ref{thm:coincide}(ii), 
and (i) above.
\end{proof}

\section{A complement for $R$ in 
  ${\cal D} \otimes {\cal D}^* \otimes {\cal D}$}

\indent
With reference to Notation \ref{notation} and Definition \ref{def:R}, 
our goal in this section is to show that the 
elements 
$\{E_i \otimes \tau^*_{j-i}(A^*) \otimes E_j \,|\, 0 \leq i\leq j \leq d\}$
form a basis for a complement of $R$ in 
${\cal D} \otimes {\cal D}^* \otimes {\cal D}$.
We begin with a slightly technical lemma.

\medskip

\begin{lemma}           \label{eq:tech}    \samepage
With reference to Notation {\rm \ref{notation}}
and Definition {\rm \ref{def:Rt}},
for $0 \leq t \leq d$ and
$0 \leq i <j \leq i+t \leq d$ the space 
\begin{equation}    \label{eq:auxd1}
 R_t + \text{\rm Span}\{E_h \otimes \tau^*_t(A^*) \otimes E_h \,|\, 0 \leq h < i\}  
\end{equation}
contains both
\begin{equation}    \label{eq:auxd2}
  f^t_{ij}(\th_i) E_i \otimes \tau^*_t(A^*) \otimes E_i  
  +  f^t_{ij}(\th_j) E_i \otimes \tau^*_t(A^*) \otimes E_j, 
\end{equation}
\begin{equation}    \label{eq:auxd3}
 f^t_{ij}(\th_i) E_i \otimes \tau^*_t(A^*) \otimes E_i  
  +  f^t_{ij}(\th_j) E_j \otimes \tau^*_t(A^*) \otimes E_i,
\end{equation}
where 
$f^t_{ij} = \prod_{h=i+1,\,h \neq j}^{i+t} (\lambda-\th_h)$.
\end{lemma}

\begin{proof}
We fix $t$ and show by induction on $i=0,1,\ldots, d-t$
that each of \eqref{eq:auxd2}, \eqref{eq:auxd3} is contained in 
\eqref{eq:auxd1} for $i<j\leq i+t$.
Concerning \eqref{eq:auxd2}, in the equation
$f^t_{ij}(A) = \sum_{n=0}^d f^t_{ij}(\th_n)E_n$
we tensor each term on the left by $E_i \otimes \tau^*_t(A^*)$ to get
\begin{equation}    \label{eq:auxd4}
 E_i \otimes \tau^*_t(A^*) \otimes f^t_{ij}(A)  
 =  \sum_{n=0}^d  f^t_{ij}(\th_n) E_i \otimes \tau^*_t(A^*) \otimes E_n.
\end{equation}
We examine the terms in \eqref{eq:auxd4}. 
The left side of \eqref{eq:auxd4} is in $R_t$ by \eqref{eq:T1} and since 
$f^t_{ij}$ has degree $t-1$.
For $0 \leq n \leq d$ consider the $n$-summand on the right in \eqref{eq:auxd4}.
First assume $0 \leq n < i-t$. Then the $n$-summand is in $R_t$ by \eqref{eq:S3}.
Next assume $i-t\leq  n < i$. 
By \eqref{eq:auxd3} and induction,
\begin{align*}
 f^t_{ni}(\th_n) E_n \otimes \tau^*_t(A^*) &\otimes E_n + 
     f^t_{ni}(\th_i) E_i \otimes \tau^*_t(A^*) \otimes E_n    \\
 & \in R_t + 
     \text{Span}\{E_h \otimes \tau^*_t(A^*) \otimes E_h \,|\, 0 \leq h < n\}.
\end{align*}
By this and since $f^t_{ni}(\th_i)$ is nonzero,
\begin{equation}    \label{eq:auxd5}
 E_i \otimes \tau^*_t(A^*) \otimes E_n
 \in R_t + 
 \text{Span}\{E_h \otimes \tau^*_t(A^*) \otimes E_h \,|\, 0 \leq h \leq n\}.
\end{equation}
Therefore our $n$-summand is in \eqref{eq:auxd1}.
Next assume $i+1\leq n \leq i+t$, $n\not=j$. Then the $n$-summand is $0$
since $f^t_{ij}(\th_n)=0$. 
Next assume $i+t<n\leq d$. Then the $n$-summand is in $R_t$ by \eqref{eq:S3}.
By these comments the expression \eqref{eq:auxd2} is 
contained in \eqref{eq:auxd1}.
By this and Theorem \ref{thm:dd}(i) the expression \eqref{eq:auxd3} is 
contained in \eqref{eq:auxd1}.
\end{proof}

\begin{proposition}   \label{prop:basis}  \samepage
With reference to Notation {\rm \ref{notation}} and
Definition {\rm \ref{def:Rt}},
for $0 \leq t \leq d$ the vectors 
\begin{equation}   \label{eq:basis}
  \{E_i \otimes \tau^*_t(A^*) \otimes E_{i+t} \,|\, 
               0 \leq i \leq d-t\}
\end{equation}
form a basis for a complement of ${R}_t$
in ${\cal D}\otimes \tau^*_t(A^*) \otimes {\cal D}$.
\end{proposition}

\begin{proof}
Consider the quotient vector space
\[
   V_t = {\cal D} \otimes \tau^*_t(A^*)  \otimes {\cal D} /R_t.
\] 
We show the vectors
\begin{equation}    \label{eq:auxe1}
  E_i \otimes \tau^*_t(A^*) \otimes E_{i+t} + R_t 
             \qquad\qquad(0 \leq i \leq d-t)
\end{equation}
form a basis for $V_t$. 
By Theorem \ref{thm:basis1} the vectors
\begin{equation}     \label{eq:auxe2}
 E_i \otimes \tau^*_t(A^*) \otimes E_i + R_t  
      \qquad\qquad      (0 \leq i \leq d-t)  
\end{equation}
form a basis for $V_t$. 
Write the elements \eqref{eq:auxe1} in terms of the basis \eqref{eq:auxe2}.
By Lemma \ref{eq:tech} the resulting coefficient  matrix is upper triangular 
with all diagonal entries nonzero. 
Therefore \eqref{eq:auxe1} is a basis for $V_t$ and the result follows.
\end{proof}

\begin{theorem}       \label{thm:basis2}
With reference to Notation {\rm \ref{notation}} and
Definition {\rm \ref{def:R}}
the vectors
\begin{equation}    \label{eq:basis2}
  \{E_i \otimes \tau^*_{j-i}(A^*) \otimes E_j \,|\, 0 \leq i\leq j\leq d \} 
\end{equation}
form a basis for a complement of $R$ in 
${\cal D} \otimes {\cal D}^* \otimes {\cal D}$.
\end{theorem}

\begin{proof}
The set \eqref{eq:basis2} is the disjoint union over $t=0,1,\ldots, d$
of the sets \eqref{eq:basis}. 
The result follows in view of
Lemma \ref{lem:DDsD},
Theorem \ref{thm:coincide}(ii), and Proposition \ref{prop:basis}
\end{proof}

\section{The space ${\cal D} \otimes E^*_0 \otimes {\cal D}$}

\indent
With reference to Notation \ref{notation} and Definition \ref{def:R},
in this section we show that the elements 
$\{E_i \otimes E^*_0 \otimes E_j \,|\, 0 \leq i \leq j \leq d\}$
form a basis for a complement of $R$ in 
${\cal D} \otimes {\cal D}^* \otimes {\cal D}$.
We will use the following lemma.

\medskip

\begin{lemma}    \label{lem:new}    \samepage
With reference to Notation {\rm \ref{notation}}, for $0 \leq t \leq d$
and $0 \leq i \leq d-t$ the element
\[
  E_i \otimes \tau^*_t(A^*) \otimes E_{i+t} 
     -  (\th^*_0-\th^*_1)(\th^*_0-\th^*_2) \cdots (\th^*_0-\th^*_t)
                          E_i \otimes E^*_0 \otimes E_{i+t} 
\]
is contained in
\begin{equation}   \label{eq:auxf1}
    R  +  \sum_{n=t+1}^d  {\cal D} \otimes \tau^*_n(A^*) \otimes {\cal D}.
\end{equation}
\end{lemma}

\begin{proof}
Applying the equation on the right in \eqref{eq:tauiA} to $\Phi^*$,
\[
 E^*_0 =  \sum_{n=0}^d 
      \frac{\eta^*_{d-n}(\th^*_0) \tau^*_n(A^*)}{\eta^*_d(\th^*_0)}.
\]
In this equation we tensor each term on the left by $E_i$
and on the right by $E_{i+t}$ to get
\begin{equation}    \label{eq:auxf2}
 E_i E^*_0 E_{i+t} =  \sum_{n=0}^d 
      \frac{\eta^*_{d-n}(\th^*_0)}{\eta^*_d(\th^*_0)}
               E_i \otimes \tau^*_n(A^*) \otimes E_{i+t}.
\end{equation}
For $0 \leq n \leq d$ consider the $n$-summand on the right in \eqref{eq:auxf2}.
For $0 \leq n \leq t-1$ the $n$-summand is in $R$ by \eqref{eq:S3}. 
For $t +1 \leq n \leq d$ the $n$-summand is in \eqref{eq:auxf1} by construction.
The result follows from these comments and since
\[
 \eta^*_d(\th^*_0) = 
  (\th^*_0-\th^*_1)(\th^*_0-\th^*_2) \cdots (\th^*_0-\th^*_t) 
   \eta^*_{d-t}(\th^*_0).
\]
\end{proof}

\medskip

\begin{theorem}             \label{thm:basis3}
With reference to Notation {\rm \ref{notation}} and
Definition {\rm \ref{def:R}} the vectors
\begin{equation}    \label{eq:basis4}
  \{E_i \otimes E^*_0 \otimes E_j \,|\, 0 \leq i\leq j\leq d\}
\end{equation}
form a basis for a complement of $R$ in 
${\cal D} \otimes {\cal D}^* \otimes {\cal D}$.
\end{theorem}

\begin{proof}
The cardinality of the set \eqref{eq:basis4} is $(d+1)(d+2)/2$,
and by Corollary \ref{cor:dim}(iv) this is the codimension of 
$R$ in ${\cal D} \otimes {\cal D}^* \otimes {\cal D}$.
Therefore, it suffices to show that $R$ and the elements \eqref{eq:basis4} 
together span ${\cal D} \otimes {\cal D}^* \otimes {\cal D}$.
Let $S$ denote the subspace of ${\cal D} \otimes {\cal D}^* \otimes {\cal D}$ 
spanned by $R$ and the elements \eqref{eq:basis4}. 
To show that $S= {\cal D} \otimes {\cal D}^* \otimes {\cal D}$
we show 
  ${\cal D} \otimes \tau^*_t(A^*) \otimes {\cal D}  \subseteq S$
for $0 \leq t \leq d$. 
We show this by induction on $t=d,d-1, \ldots, 0$.
Let $t$ be given. 
By Proposition \ref{prop:basis},
\[ 
 {\cal D} \otimes \tau^*_t(A^*) \otimes {\cal D}
  =  R_t + 
   \text{Span}\{E_i \otimes \tau^*_t(A^*) \otimes E_{i+t}\,|\,0 \leq i \leq d-t\}.
\]
By construction $R_t \subseteq R \subseteq S$. 
For $0 \leq i \leq d-t$ we have
 $E_i \otimes \tau^*_t(A^*) \otimes E_{i+t} \in S$
by Lemma \ref{lem:new} and induction on $t$. 
By these comments
 ${\cal D} \otimes \tau^*_t(A^*) \otimes {\cal D} \subseteq S$ 
and the result follows.
\end{proof}

\section{The proof of Theorem \ref{thm:main}}

\indent
Using the results in earlier sections we can now easily
prove Theorem \ref{thm:main}.

\medskip

\begin{proofof}{Theorem \ref{thm:main}}
(i): 
By Definition \ref{def:pi} the image
$\pi({\cal D} \otimes {\cal D}^* \otimes {\cal D})$ is the span of
$E^*_0{\cal D}{\cal D}^*{\cal D}E^*_0$.
Similarly the image 
$\pi({\cal D} \otimes E^*_0 \otimes {\cal D})$ is the span of
$E^*_0{\cal D}E^*_0{\cal D}E^*_0$. We show
\begin{equation}    \label{eq:auxc}
 \pi({\cal D} \otimes {\cal D}^* \otimes {\cal D})
  = \pi({\cal D} \otimes E^*_0 \otimes {\cal D}). 
\end{equation}
Let $C$ denote the subspace of 
${\cal D} \otimes {\cal D}^* \otimes {\cal D}$ spanned by the 
elements \eqref{eq:basis4}.
By Theorem 9.3 
${\cal D} \otimes {\cal D}^* \otimes {\cal D}$ is the direct sum $C + R$.
By the construction $C$ is contained in 
${\cal D} \otimes E^*_0 \otimes {\cal D}$.
By Lemma \ref{lem:kernel} the space $R$ is contained in the kernel of $\pi$. 
Therefore
\[
  {\cal D} \otimes {\cal D}^* \otimes {\cal D}
  = {\cal D} \otimes E^*_0 \otimes {\cal D} + \text{Ker}(\pi).
\]
Applying $\pi$ to this equation we get \eqref{eq:auxc} and the result follows.

(ii):
For $X, Y \in {\cal D}$ we show
$E^*_0XE^*_0$, $E^*_0YE^*_0$ commute.
By Theorem \ref{thm:dd}(ii),
\[
  X \otimes E^*_0 \otimes Y  - Y \otimes E^*_0 \otimes X \in R.
\]
In the above line we apply the map $\pi$
and use Lemma \ref{lem:kernel} to find
\[
    E^*_0 X E^*_0 Y E^*_0  = E^*_0 Y  E^*_0  X E^*_0.
\]
By this and since $E^{*2}_0=E^*_0$ the elements $E^*_0XE^*_0$,  
$E^*_0YE^*_0$  commute.
\end{proofof}

\section{Acknowledgements}

\indent
The second author gratefully acknowledges many conversations with 
Tatsuro Ito (Ka\-na\-za\-wa University) on the general subject of this paper. 
The resulting insights lead directly to the paper, and consequently
we feel that Ito deserves to be a coauthor.
We offered him this coauthorship but he declined.

\bigskip

{\small

\bibliographystyle{plain}

}

\bigskip\bigskip\noindent
Kazumasa Nomura\\
College of Liberal Arts and Sciences\\
Tokyo Medical and Dental University\\
Kohnodai, Ichikawa, 272-0827 Japan\\
email: knomura@pop11.odn.ne.jp

\bigskip\noindent
Paul Terwilliger\\
Department of Mathematics\\
University of Wisconsin\\
480 Lincoln Drive\\ 
Madison, Wisconsin, 53706 USA\\
email: terwilli@math.wisc.edu

\bigskip\noindent
{\bf Keywords.}
Leonard pair, tridiagonal pair, $q$-Racah polynomial, orthogonal polynomial.

\noindent
{\bf 2000 Mathematics Subject Classification}.
05E35, 05E30, 33C45, 33D45.

\end{document}